\documentclass{amsart}
\usepackage{amsmath}
\usepackage{amsfonts}
\usepackage{amssymb}
\usepackage{color}
\usepackage{graphics,graphicx}
\usepackage{hyperref}
\usepackage{amssymb,latexsym,amsthm,xypic}

\newtheorem{thm}{Theorem}[section]

\newtheorem{lem}[thm]{Lemma}
\newtheorem{prop}[thm]{Proposition}

\theoremstyle{definition}

\theoremstyle{remark}

\newcommand{\be}{\begin{equation}}
\newcommand{\ee}{\end{equation}}
\newcommand{\bea}{\begin{eqnarray}}
\newcommand{\eea}{\end{eqnarray}}
\newcommand{\ben}{\begin{eqnarray*}}
\newcommand{\een}{\end{eqnarray*}}
\newcommand{\bt}{\begin{split}}
\newcommand{\et}{\end{split}}
\newcommand{\bet}{\begin{equation}}
\newcommand{\mc}{\mathbb{C}}

\newcommand{\ra}{\rightarrow}

\begin{document}

\title[Minimum Principle]{A new proof of Kiselman's minimum principle for plurisubharmonic functions}
\date{}
\author[F. Deng]{Fusheng Deng}
\address{Fusheng Deng: \ School of Mathematical Sciences, University of Chinese Academy of Sciences\\ Beijing 100049, P. R. China}
\email{fshdeng@ucas.ac.cn}
\author[Z. Wang]{Zhiwei Wang}
\address{ Zhiwe Wang: \ School
of Mathematical Sciences\\Beijing Normal University\\Beijing\\ 100875\\ P. R. China}
\email{zhiwei@bnu.edu.cn}
\author[L. Zhang]{Liyou Zhang}
\address{ Liyou Zhang: \ School of Mathematical Sciences\\Capital Normal University\\Beijing \\100048\\
P. R.  China}
\email{zhangly@cnu.edu.cn}
\author[X. Zhou]{Xiangyu Zhou}
\address{Xiangyu Zhou:\ Institute of Mathematics, AMSS, and Hua Loo-Keng Key Laboratory of Mathematics, Chinese Academy of Sciences, Beijing 100190,
China; School of Mathematical Sciences, University of Chinese
Academy of Sciences, Beijing 100049, China}
\email{xyzhou@math.ac.cn}

\begin{abstract}
We give a new proof of Kiselman's minimum principle for plurisubharmonic functions,
based on Ohsawa-Takegoshi extension theorem.
\end{abstract}
\maketitle

\section{Introduction}\label{sec:intro}
The aim of this note is to provide a new proof of Kiselman's minimum principle for plurisubharmonic functions,
based on the Ohsawa-Takegoshi extension theorem about holomorphic functions.

\begin{thm}[\cite{Kiselman78}]\label{thm-intr:Kiselman}
Let $\Omega\subset\mc_t^r\times\mc_z^n$ be a pseudoconvex domain and let $p:\Omega\ra U:=p(\Omega)\subset\mc^r$
be the natural projection from $\Omega$ to $\mc^r$.
Let $\varphi$ be a plurisubharmonic function on $\Omega$.
Assume that all fibers $\Omega_t:=p^{-1}(t)$ $(t\in U)$ are connected tube domains
and $\varphi(t,z)$ is independent of the image part of $z$ for all $(t,z)\in \Omega$.
Then the function $\varphi^*$ defined as
$$\varphi^*(t):=\inf_{z\in\Omega_t}\varphi(t,z)$$
is a plurisubharmonic function on $U$.
\end{thm}

Let $T^n$ be the $n$-dimensional torus group.
The natural action of $T^n$ on $\mc^n$ is given by
$(e^{i\theta_1},\cdots, e^{i\theta_n})\cdot(z_1,\cdots, z_n):=(e^{i\theta_1}z_1, \cdots, e^{i\theta_n}z_n)$.
A domain $D$ in $\mc^n$ is called a Reinhardt domain if it is invariant under the action of $T^n$.

Taking exponential map, it is obvious that Theorem \ref{thm-intr:Kiselman} is a consequence of the following

\begin{thm}\label{thm-intr:Reinhardt case}
Let $\Omega\subset\mc_t^r\times\mc_z^n$ be a pseudoconvex domain and let $p:\Omega\ra U:=p(\Omega)\subset\mc^r$
be the natural projection from $\Omega$ to $\mc^r$ such that all fibers $\Omega_t:=p^{-1}(t)$ ($t\in U$) are (connected) Reinhardt domains.
Let $\varphi$ be a plurisubharmonic function on $\Omega$ such that $\varphi(t, \alpha z)=\varphi(t,z)$ for $\alpha\in T^n$.
then the function $\varphi^*$ defined as
$$\varphi^*(t):=\inf_{z\in\Omega_t}\varphi(t,z)$$
is a plurisubharmonic function on $U$.
\end{thm}

The argument in the proof of Theorem \ref{thm-intr:Reinhardt case} can be generalized to
more general settings considered in \cite{Deng-Zhou-Zhang 14, Deng-Zhou-Zhang15}.
For example, in the same way one can show the following

\begin{thm}\label{thm-intr:main}
Let $\Omega\subset\mc_t^r\times\mc_z^n$ be a pseudoconvex domain and let $p:\Omega\ra U:=p(\Omega)\subset\mc^r$
be the natural projection from $\Omega$ to $\mc^r$ such that all fibers $\Omega_t:=p^{-1}(t)$ ($t\in U$) are connected.
Let $\varphi$ be a plurisubharmonic function on $\Omega$.
Assume that there exists a compact Lie group $K$ acting on $\mc^n$ holomorphically such that:
\begin{itemize}
\item[(a)] the action of $K$ on $\mc^n$ preserves the Lebesgue measure;
\item[(b)] all fibers $\Omega_t:=p^{-1}(t)$ ($t\in U$) are invariant under the action of $K$;
\item[(c)] $K$-invariant holomorphic functions on $\Omega_t$ are constant for $t\in U$;
\item[(d)] $\varphi$ is invariant under the action of $K$.
\end{itemize}
then the function $\varphi^*$ defined as
$$\varphi^*(t):=\inf_{z\in\Omega_t}\varphi(t,z)$$
is a plurisubharmonic function on $U$.
\end{thm}

Our proof of Theorem \ref{thm-intr:Reinhardt case} is inspired by the method
of Demailly on regularization of plurisubharmonic functions \cite{Dem92}.
Applying Ohsawa-Takegoshi extension theorem to extending holomorphic functions from discrete points,
Demailly shows that a plurisubharmonic function can be approximated by certain Bergman kernels.
In the recent work \cite{DWZZ}, the idea was developed to give a new characterization of plurisubharmonic functions.

It is natural to ask what can we get if we apply Demailly's idea to
extending holomorphic functions from submanifolds of positive dimension.
In this note, we show that this can lead to Kiselman's minimum principle for plurisubharmonic functions,
namely the above theorems and their generalizations.

The above theorems were proved by Berndtsson in \cite{Berndtsson98}(see also \cite{Deng-Zhou-Zhang 14}) by
showing a Prekopa-type theorem for plurisubharmonic functions.
The method in this note is quite different from those in \cite{Kiselman78}\cite{Berndtsson98} and \cite{Deng-Zhou-Zhang 14}.

\subsection*{Acknowledgements}
The authors are partially supported by NSFC grants.

\section{Proof of the main result}\label{sec:proof of main result}
In this section we give the proof of Theorem \ref{thm-intr:main}.
We need the Ohsawa-Takegoshi extension theorem for holomorphic functions.

\begin{lem}[\cite{OT1}]\label{lem:OT}
Let $\Omega\subset\mc_t^r\times\mc_z^n$ be a bounded pseudoconvex domain and let $p:\Omega\ra U:=p(\Omega)\subset\mc^r$
be the natural projection from $\Omega$ to $\mc^r$.
Let $\varphi$ be a plurisubharmonic function on $\Omega$.
Then for any $t\in U$ and for any holomorphic function $f$ on $\Omega_t:=p^{-1}(t)$,
there exists a holomorphic function $F$ on $\Omega$ such that $F|_{\Omega_t}=f$, and
$$\int_\Omega|F|^2e^{-\varphi}\leq C\int_{\Omega_t}|f|^2e^{-\varphi_t},$$
where $\varphi_t(z)=\varphi(t,z)$ and $C$ is a constant independent of $\varphi$, $t$, and $f$.
\end{lem}

We will prove the following Proposition \ref{prop:product case} and show that
Theorem \ref{thm-intr:main} is a consequence of it.

\begin{prop}\label{prop:product case}
Let $D\subset\mc_z^n$ be a bounded pseudoconvex Reinhardt domain and $U\subset \mc_t^r$,
and let $\varphi$ be a plurisubharmonic function on $\Omega:=U\times D$
which is continuous on the closure of $\Omega$.
Assume that $\varphi(t, \alpha z)=\varphi(t,z)$ for $\alpha\in T^n$.
Then the function $\varphi^*$ on $U$ defined by
$$\varphi^*(t):=\inf_{z\in D}\varphi(t,z)$$
is plurisubharmonic.
\end{prop}

\begin{proof}
Since $\varphi^*$ is upper-semicontinuous,
we can assume that $U$ is a planar domain.

For any positive integer $m$,
let $H^2(\Omega, m\varphi)=\{f\in\mathcal O(\Omega); \|f\|_m:=\int_\Omega|f|^2e^{-m\varphi}<\infty\}$
be the Hilbert space of holomorphic functions on $\Omega$ which are square integrable with respect
to the weight $e^{-m\varphi}$, and let $H^2(\Omega, m\varphi)^I$ be the subspace of $H^2(\Omega, m\varphi)$
consisting of $T^n$-invariant holomorphic functions.
We define $E^m_t=H^2(D,m\varphi_t)^I$ with $t\in U$ in the same way.

Let $\tilde E^m=U\times H^2(\Omega, m\varphi)^I$ and $E^m=\coprod_{t\in U}H^2(D,m\varphi_t)^I$.
Then $\tilde E^m$ is a trivial and flat holomorphic hermitian vector bundle (of infinite rank) over $U$,
and $E^m$ is a trivial holomorphic line bundle over $U$.
Let $\mathbb{I}$ be the canonical holomorphic frame of $E^m$
which restricts to $\Omega_t:=\{t\}\times D$ the constant function with value 1, for all $t\in U$.
Then a holomorphic section of $E$ over $U$ can be naturally identified with a holomorphic function on $U$.

There is a canonical holomorphic bundle morphism $\pi:\tilde E^m\rightarrow E^m$,
with $\pi_t:H^2(\Omega, m\varphi)^I\rightarrow E_t$ given by $f\mapsto f|_{\Omega_t}$.
It is clear that $\pi$ is surjective and hence induces a (singular) hermitian metric on $E^m$,
which is denoted by $h_m$.
Explicitly, $h_m(\mathbb I_t)=\inf\{\|f\|_m; f\in H^2(\Omega, m\varphi)^I, f|_{\Omega_t}\equiv 1\}$.

Let $h^*$ be the metric on $E:=E^1$ given by $h^*(\mathbb{I})=e^{-\varphi^*}$.
Let $h_m(\mathbb{I})=e^{-m\varphi_m}$.
Since $\tilde E^m$ is a flat bundle and $(E^m, h_m)$ is a quotient bundle of $\tilde E^m$,
$(E^m, h_m)$ has semipositive curvature current
and hence $\varphi_m(z)$ is a subharmonic function on $U$ for all $m$.
We want to show that $\varphi_m$ converges to $\varphi^*$ in some sense as $m\rightarrow\infty$.

For fix $t\in U$, let $f\in H^2(\Omega, m\varphi)$ be the function with minimal norm such that
$f|_{\Omega_t}\equiv 1$.
By the uniqueness of the minimal element and the $T^n$-invariance of $\varphi$,
it is clear that $f\in H^2(\Omega, m\varphi)^I$.
By Lemma \ref{lem:OT}, we have the estimate
$$\int_\Omega|f|^2e^{-m\varphi}\leq C\int_{\Omega_t} e^{-m\varphi_z},$$
where $C$ is a constant independent of $m$ and $t$.

By definition,
\begin{equation*}
-\varphi_m(t)=\frac{1}{m}\log\left(\int_\Omega|f|^2e^{-m\varphi}\right).
\end{equation*}
So we have
\begin{equation}\label{eqn:h_m<}
-\varphi_m(t)\leq \frac{1}{m}\log\left(\int_{\Omega_t} e^{-m\varphi_t}\right)+ \frac{\log C}{m}.
\end{equation}

We apply the mean value inequality to prove another inequality.
For any $\epsilon>0$, there is $0<r<<1$ independent of $z$ such that $\varphi(t',z)\leq\varphi(t,z)+\epsilon$
for any $(t',z)\in \Omega$ with $|t-t'|<r$. By mean value inequality, we have
\begin{equation*}
\begin{split}
     & \int_\Omega|f|^2e^{-m\varphi}\\
\geq & \int_{\Delta(t,r)\times D}|f|^2e^{-m\varphi}\\
\geq & \int_D\left(\int_{\Delta(t,r)}|f(\tau,z)|^2d\mu_\tau\right)e^{-m(\varphi_z+\epsilon)}d\mu_z\\
\geq & \pi r^2\int_D e^{-m(\varphi_z+\epsilon)},
\end{split}
\end{equation*}
where $\Delta(t,r)=\{\tau\in\mc; |\tau-t|<r\}$, and $d\mu_\tau$ and $d\mu_z$ are Lebesgue measures on $U$ and $D$ respectively.
This implies
\begin{equation}\label{eqn:h_m>}
-\varphi_m(t)\geq \frac{1}{m}\log\left(\int_D e^{-m(\varphi_t+\epsilon)}\right)+ \frac{\log\pi r^2}{m}.
\end{equation}
Combing \eqref{eqn:h_m<} and \eqref{eqn:h_m>}, we have
\begin{equation}\label{eqn:bi-inequality}
-\frac{1}{m}\log\left(\int_D e^{-m\varphi_t}\right)-\frac{\log C}{m}
\leq\varphi_m\leq
-\frac{1}{m}\log\left(\int_D e^{-m(\varphi_t+\epsilon)}\right)-\frac{\log\pi r^2}{m},
\end{equation}
which implies $\varphi_m$ converges to $\varphi$ pointwise on $\Delta$.

Let $\tilde\varphi_m$ be the upper semicontinuous envelope of $\sup_{j\geq m}\varphi_j$ and let $\tilde\varphi^*$ be the limit of $\tilde\varphi_m$.
Then $\tilde\varphi^*$ is plurisubharmonic and hence it suffices to prove $\varphi^*=\tilde\varphi^*$.
By the first inequality in \eqref{eqn:bi-inequality}, it is obvious that $\tilde\varphi^*\geq \varphi^*$.
Note that $\varphi$ is assumed to be uniformly continuous on the closure of $\Omega$,
the last term in \eqref{eqn:bi-inequality} converges uniformly on $U$ to $\varphi^*+\epsilon$ as $m\ra\infty$.
So we have $\tilde\varphi^*\leq \varphi^*+2\epsilon$ if $m$ is large enough.
Note that $\epsilon$ is arbitrary, $\tilde\varphi^*\leq\varphi^*$ and hence $\varphi^*=\tilde\varphi^*$.
\end{proof}

We now explain how to deduce Theorem \ref{thm-intr:Reinhardt case} from Proposition \ref{prop:product case}.
The argument is inspired by the idea in \cite{Berndtsson98}.
By approximation, we can assume that $\varphi$ is continuous.
By replacing $\Omega$ by a smaller domain, we can assume that $\varphi$ is defined on some neighborhood of $\overline\Omega$,
and there is a smooth plurisubharmonic function $\rho$ on some neighborhood of $\overline\Omega$
such that $\Omega$ is given by $\rho< 0$.
By averaging, $\rho$ can be taken to be invariant under the action of $T^n$.
Since the theorem is local for $U$,
we can contract $U$ and find a pseudoconvex Reinhardt domain $V$ such that $\Omega\subset U\times V$
and $U\times V$ lies in the domain of definition of $\varphi$ and $\rho$.
Applying Proposition \ref{prop:product case} to the domain $U\times V$ and the function $\varphi+M\max\{\rho,0\}$,
and letting $M\ra +\infty$, we get Theorem \ref{thm-intr:Reinhardt case}.

\bibliographystyle{amsplain}

\maketitle

\end{document}